\documentclass[11pt,reqno]{amsart}
\usepackage{graphicx}
\usepackage{epsfig}
\usepackage{amssymb,amsthm,amsmath,amsfonts,amstext,mathrsfs}
\usepackage{latexsym}
\usepackage{url}

\newtheorem{lem}{Lemma}

\newtheorem{conj}{Conjecture}
\newtheorem{assum}{Assumption}

\input xy
\xyoption{all}

\def\d{\,{\rm{d}}}

\def\o{\,{\omega}}

\title[Series for the eigenvalues of the GKW operator]
{Recursive construction of a series converging to the eigenvalues of the Gauss-Kuzmin-Wirsing operator}
\author[Giedrius Alkauskas]{Giedrius Alkauskas}
\begin{document}
\begin{abstract} Based on the technique previously developed by the author,
we present a conjecture which claims that the reciprocal of the
$n$th largest (in absolute value) eigenvalue of the
Gauss-Kuzmin-Wirsing operator is equal to the sum of a certain
infinite series. This series is constructed recurrently. It consists
of rational functions with integer coefficients in two variables
$\mathbf{X},\mathbf{Y}$, specialized at $\mathbf{X}=n$ and
$\mathbf{Y}=2^{n}$. This gives a strong evidence to the conjecture
of Mayer and Roepstorff that eigenvalues have alternating sign.
Further, a very similar recursion yields a series for the dominant
eigenvalue of the Mayer-Ruelle operator.
\end{abstract}
\maketitle
\begin{center}
\end{center}
\begin{center}
\end{center}
\section{Introduction and conjecture}
\footnotetext[1]{The author gratefully acknowledges support from the
Austrian Science Fund (FWF) under the project Nr. P20847-N18.} Let
$\mathbb{D}$ be a disc $\{x\in\mathbb{C}:|x-1|<\frac{3}{2}\}$. Let
$\mathbf{V}$ be a Banach space of functions which are analytic in
$\mathbb{D}$ and are continuous in its closure. We equip this space
with the supremum norm. \emph{The Gauss-Kuzmin-Wirsing operator} is
defined for functions $f\in\mathbf{V}$ by
\cite{khinchin,knuth,wirsing}
\begin{eqnarray}
\mathcal{L}[f(t)](x)=\sum\limits_{m=1}^{\infty}\frac{1}{(x+m)^2}f\Big{(}\frac{1}{x+m}\Big{)}.\label{gkw}
\end{eqnarray}
Our main interest in this paper is the \emph{point spectrum} of this
operator. In fact, the operator $\mathcal{K}[\hat{f}]$, defined by
\begin{eqnarray*}
\mathcal{K}[\hat{f}(t)](x)=\int\limits_{0}^{\infty}\frac{J_{1}(2\sqrt{xy})}{\sqrt{(e^{x}-1)(e^{y}-1)}}\cdot
\hat{f}(y)\d y,
\end{eqnarray*}
for $\hat{f}$ belonging to the Hilbert space
$L^2(\mathbb{R}_{+},m)$, $\d m(y)=\frac{y}{e^{y}-1}\d y$, has the
same point spectrum \cite{babenko} (it is easy to pass from $f$ to
$\hat{f}$ using the Borel transform). As was shown in
\cite{mayer1,mayer2}, the latter operator is compact, is of trace
class and it is nuclear of order $0$. Thus, it possesses the
eigenvalues $\lambda_{n}$, $n\in\mathbb{N}$, which are real numbers,
$|\lambda_{n}|\geq |\lambda_{n+1}|$, $\lambda_{1}=1$, and
$\sum_{n=1}^{\infty}|\lambda_{n}|^{\epsilon}<+\infty$ for every
$\epsilon>0$.\\

The Gauss-Kuzmin-Wirsing operator is intricately related with the
Gauss map $F(x)=\{1/x\}$, $x\in(0,1]$, $F(0)=0$ (here $\{\star\}$
stands for the fractional part). Let $F^{(1)}=F$, and
$F^{(k)}=F\circ F^{(k-1)}$ for $k\geq 2$. As is now well-known (due
to important contributions by Gauss, Kuzmin, L\'{e}vy, Wirsing,
Babenko, Babenko and Jur'ev, Mayer) we have
\begin{eqnarray*}
\mu(a\in[0,1]:F^{(k)}(a)<x)=\frac{\log(1+x)}{\log2}+\sum\limits_{n=2}^{\infty}\lambda^{k}_{n}\Phi_{n}(x).
\end{eqnarray*}
Here $\mu(\star)$ stands for the Lebesgue measure, and for each
$n\geq 2$, the function $\Phi_{n}(x)$ is defined in the cut plane
$x\in\mathbb{C}\setminus(-\infty,-1]$, it satisfies the boundary
conditions $\Phi_{n}(0)=\Phi_{n}(1)=0$ ($n\geq 2$), the regularity
condition $\Phi'_{n}(x)\rightarrow 0$ if ${\rm
dist}(x,(-\infty,-1])\rightarrow\infty$, and the functional equation
\begin{eqnarray}
\Phi_{n}(x+1)-\Phi_{n}(x)=\frac{1}{\lambda_{n}}\cdot\Phi_{n}\Big{(}\frac{1}{x+1}\Big{)}.\label{vienas}
\end{eqnarray}
Thus, $\Phi_{1}(x)=\frac{\log(1+x)}{\log 2}$. The eigenfunctions of $\mathcal{L}$ are then given by $\Phi'_{n}(x)$, $n\in\mathbb{N}$. More details can be found in \cite{babenko,knuth,wirsing}. \\

The nature of the eigenvalues $\lambda_{n}$ is unknown. It is unanimously believed that these constants are unrelated to other most important constants in mathematics; in particular, it is expected that they are neither algebraic numbers nor periods (periods are numbers like $\pi$, Catalan constant, $\zeta(3)$, and so on). Moreover, though now more that $480$ digits of $\lambda_{2}$ have been calculated \cite{briggs}, there is no rigorous result known which guarantees that the digits of $\lambda_{2}$ we calculate are the true ones. Though, as noted in \cite{flajolet1}, one can (theoretically) get certificates at least for $\lambda_{2}$. Concerning calculations of values of $\lambda_{n}$ for $n\geq 2$, one can only trust heuristic arguments, which are absolutely likely to be true, as numerical calculations suggest \cite{flajolet1,knuth,macleod,zagier1}. We note that the first few digits of $\lambda_{2}$ can be calculated rigorously \cite{mayer1}.\\
\indent On the other hand, the trace of the operator $\mathcal{L}$ can be given explicitly. As an aside, there exist formulas for $\mathrm{Tr}(\mathcal{L}^{k})$ for $k\in\mathbb{N}$ \cite{zagier2,mayer1,mayer2}. These formulas are crucial in Mayer's proof \cite{mayer3} that the Fredholm determinant $\mathrm{det}(1-\mathcal{L}_{2s}^{2})$ is equal to the Selberg zeta function for the full modular group. Here $\mathcal{L}_{2s}$ is \emph{the Mayer-Ruelle operator}, which is given by (\ref{mr}) below. As was shown in \cite{mayer1} (see also \cite{flajolet2,finch}), we have
\begin{eqnarray*}
\mathrm{Tr}(\mathcal{L})=\sum\limits_{n=1}^{\infty}\lambda_{n}=\int\limits_{0}^{\infty}\frac{J_{1}(2x)}{e^{x}-1}\d x
=\sum\limits_{m=1}^{\infty}\frac{1}{1+\xi_{m}^{2}};\text{ here }\xi_{m}=\frac{m+\sqrt{m^2+4}}{2}.
\end{eqnarray*}
The last sum can by expanded in terms of inverse powers of $m$. Thus, this implies \cite{flajolet2,finch}
\begin{eqnarray*}
\mathrm{Tr}(\mathcal{L})=\frac{1}{2}-\frac{1}{2\sqrt{5}}+\frac{1}{2}\sum\limits_{k=1}^{\infty}(-1)^{k-1}\binom{2k}{k}(\zeta(2k)-1)=0.7711255236_{+}.
\end{eqnarray*}
The constants $\lambda_{n}$ have received a considerable amount of attention in recent decades. Nevertheless, there are three outstanding unresolved problems. As was said before, we henceforth arrange the eigenvalues according to their absolute value $|\lambda_{1}|\geq|\lambda_{2}|\geq\cdots$. Of course, in case $\lambda_{n}=\pm\lambda_{n+1}$ for some $n$, this arrangement is not uniquely defined. Despite this, we have
\begin{conj}The following three statements are true:
\begin{itemize}
\item{{\rm Simplicity}}. The eigenvalues are simple. Moreover, $|\lambda_{n}|>|\lambda_{n+1}|$.
\item{{\rm Sign}}. The eigenvalues have alternating sign: $(-1)^{n+1}\lambda_{n}>0$.
\item{{\rm Ratio}}. There exists a limit $\lim\limits_{n\rightarrow\infty}\frac{\lambda_{n}}{\lambda_{n+1}}=-\frac{3+\sqrt{5}}{2}$.
\end{itemize}
\end{conj}
\noindent The first conjecture was raised by Babenko \cite{babenko},
the second conjecture can be attributed to Mayer and Roepstorff
\cite{mayer2}, and the last one most likely was raised by MacLeod
\cite{macleod}, and seconded by Flajolet and Vall\'{e}e
\cite{flajolet1}. It has the following explanation. The spectrum of
the operator
\begin{eqnarray*}
\mathcal{L}_{0}[f(t)](x)=\frac{1}{(x+1)^{2}}f\Big{(}\frac{1}{x+1}\Big{)},\quad f\in\mathbf{V},
\end{eqnarray*}
is given by $(-1)^{n}\phi^{-2n-2}$, $n\in\mathbb{N}_{0}$, where $\phi=\frac{\sqrt{5}+1}{2}$. It is expected that the terms in (\ref{gkw}) for $m\geq 2$ act only as small perturbations to $\mathcal{L}_{0}$. Of course, the ``Sign" and ``Simplicity" conjectures follow from the ``Ratio" conjecture for sufficiently large $n$ (provided it is effective and we can verify these conjectures for the first few values of $n$). \\

 The main target of this paper is to pose yet another conjecture. It concerns the exact values of the real numbers $\lambda_{n}$. It appears that one can interpolate the whole collection of eigenvalues $\lambda_{n}$. For this purpose, consider the following \\

\noindent \textbf{Procedure. }For integers $p,t,j\geq 0$, let us define $A_{p,t}(\mathbf{X},\mathbf{Y})$ and $\Psi_{j}(\mathbf{X},\mathbf{Y})$ (which are rational functions with integer coefficients, and in fact $\mathrm{denom}(A_{p,t}(\mathbf{X},\mathbf{Y}))$ and $\mathrm{denom}(\Psi_{j}(\mathbf{X},\mathbf{Y}))\in\mathbb{Q}[\mathbf{Y}]$), by
\begin{eqnarray*}
\Psi_{0}(\mathbf{X},\mathbf{Y})=2\mathbf{Y}-2,\quad A_{0,0}(\mathbf{X},\mathbf{Y})=1,
\end{eqnarray*}
and then recurrently by
\begin{eqnarray}
& &(2^{p}\cdot\mathbf{Y}-2^{t})\cdot A_{p,t}(\mathbf{X},\mathbf{Y})\nonumber\\
&=&\sum\limits_{k=0}^{p}\sum\limits_{j=0}^{\min\{p-k,t\}}\sum\limits_{i=0}^{t-j}
\Psi_{j}(\mathbf{X},\mathbf{Y})\cdot A_{k,i}(\mathbf{X},\mathbf{Y})\cdot\binom{\mathbf{X}+p-i-j}{p-j-k}\underline{\binom{\mathbf{X}+k-i-1}{t-i-j}}(-1)^{k+i}2^{i+j-1}\nonumber\\
&+&\sum\limits_{k=0}^{p-1}\sum\limits_{i=\max\{0,k+t-p\}}^{t-1}A_{k,i}(\mathbf{X},\mathbf{Y})\cdot
\frac{\mathbf{X}+p-t}{t-i}\binom{\mathbf{X}+p-i-1}{p-k-1}\binom{p-k-1}{p+i-k-t}(-1)^{p+k}2^{i}\nonumber\\
&+&\sum\limits_{k=0}^{p-1}A_{k,t}(\mathbf{X},\mathbf{Y})\cdot
\binom{\mathbf{X}+p-t}{p-k}(-1)^{p+k}2^{t}.\label{calc}
\end{eqnarray}
(Why we underline one binomial coefficient will be explained in the end of this section). This recursion works as follows. Fix $N\in\mathbb{N}$. We will calculate the functions $A_{p,t}$ and $\Psi_{j}$ for $0\leq p,t\leq N$ and $0\leq j\leq N$. Let $P\leq N$. Suppose we have already calculated them for $0\leq p\leq P-1$ and $0\leq t\leq N$, $0\leq j\leq P-1$. The above recursion allows to calculate $A_{P,t}$ for $0\leq t\leq P-1$. When $(p,t)=(P,P)$, this recursion \emph{a priori} contains two unknowns: $\Psi_{P}$ and $A_{P,P}$. Nevertheless, the coefficient at $A_{p,t}$ of this recurrence is equal to $(2^{p}\mathbf{Y}-\mathbf{Y}(-1)^{p-t}2^{t}-2^{t}+(-1)^{p-t}2^{t})$, which vanishes for $(p,t)=(P,P)$. Thus, since for $(p,t)=(P,P)$ the coefficient at $\Psi_{P}$ is $2^{P-1}$, we find the unique value of $\Psi_{P}$. We will later show that at this stage the function $A_{P,P}$ can be defined arbitrarily. We choose
\begin{eqnarray*}
A_{P,P}(\mathbf{X},\mathbf{Y})=0\text{ for }P\geq 1.
\end{eqnarray*}
Since now $\Psi_{P}$ and $A_{P,P}$ have been identified, we use the recurrence (\ref{calc}) to find $A_{P,t}$ for $P<t\leq N$. Consequently, according to this procedure, we can uniquely determine the rational functions $A_{p,t}$ and $\Psi_{j}$. \\
\indent At the first glace, the recurrence  (\ref{calc}) has nothing to do with continued fractions. Nevertheless, it definitely does, as is expressed by the following
\begin{conj}Let us define the rational functions $\Psi_{j}(\mathbf{X},\mathbf{Y})$, $j\geq 0$, as just described. Suppose, $1=|\lambda_{1}|\geq|\lambda_{2}|\geq|\lambda_{3}|\geq\cdots$ are the eigenvalues of the Gauss-Kuzmin-Wirsing operator, ordered by their absolute value. Then
\begin{eqnarray}
\fbox{$\displaystyle\frac{(-1)^{n+1}}{\lambda_{n}}=\sum\limits_{j=0}^{\infty}(-1)^{j}\Psi_{j}(n,2^{n})$}
\label{main}
\end{eqnarray}
and the series converges absolutely.
\label{conj2}
\end{conj}
\noindent Thus, we claim that the indexing of eigenvalues according to their absolute value is not accidental but rather a canonical one. This indexing corresponds to the degree of denominator of a rational function $G_{n}(2,z)$ (see the next section). The Table 1 in the end of this paper lists the first four rational functions $\Psi_{j}(\mathbf{X},\mathbf{Y})$. The Table 3 lists the values of $\Psi_{j}(2,4)$. According to the above conjecture, $(\sum_{j=0}^{\infty}(-1)^{j}\Psi_{j}(2,4))^{-1}=-\lambda_{2}=0.30366300289873265859_{+}$ (\emph{the Gauss-Kuzmin-Wirsing constant}). The Table 2 lists numerical values for
\begin{eqnarray*}
S_{N}(n)=\Big{(}\sum_{j=0}^{N}(-1)^{j}\Psi_{j}(n,2^n)\Big{)}^{-1}
\end{eqnarray*}
for $2\leq n\leq 5$ and some selected values of $N$. These calculations match the known digits of $\lambda_{n}$ very well (see \cite{briggs,flajolet1,macleod,zagier1} for numerical values). As in the case with the explicit series for the moments of the Minkowski question mark function \cite{alkauskas}, the series (\ref{main}) is not very useful in numerical calculations. For example, $\Psi_{30}(2,4)\approx1.5\cdot10^{-11}$, while the denominator of $\Psi_{30}(2,4)$ is approximately $2.8\cdot10^{967}$.\\
\indent Despite this drawback, we believe that the identity (\ref{main}) is significant: it shows that there exists a unifying approach towards the whole collection of eigenvalues $\{\lambda_{n},n\in\mathbb{N}\}$, apart from the approach provided
by the trace formulas for ${\rm Tr}(\mathcal{L}^{k})$. It is also important that each summand $\Psi_{j}(n,2^{n})$ is represented by a rational function with integer coefficients, hence it carries and encodes only a finite amount of information. Notwithstanding, in this paper we do not touch the most interesting question from the number-theoretic point of view; mainly, the structure of rational functions $\Psi_{j}(\mathbf{X},\mathbf{Y})$. This might be very hard, and it deserves a separate treatment. \\

The idea behind our technique which yields the Conjecture \ref{conj2} is as follows (see \cite{alkauskas} for more details). We replace \emph{the Calkin-Wilf tree} \cite{calkin} (whose permutation is \emph{the Stern-Brocot tree} and whose intersection with $[0,1]$ is \emph{the Farey tree}) by a certain tree which depends on a complex parameter $\o$, $|\o-2|\leq 1$. This tree is given by the root $1$, and each node $x$ generates two offsprings according to the rule
\begin{eqnarray}
\xymatrix @R=.5pc @C=.5pc {& & & & & & x & & & & & &\\
& & \frac{\o x}{x+1} \ar@{-}[urrrr] & & & & & & & & \frac{x+1}{\o}. \ar@{-}[ullll] & &}\label{itf}
\end{eqnarray}
When $\o=1$, we recover the Calkin-Wilf tree. For $|\o-2|\leq 1$, the closure of the direct limit of $x=1$ is a certian Julia set $\mathscr{I}_{\o}$ \cite{alkauskas}. The main idea is that this tree collapses to a single point for $\o=2$. Thus, instead of a functional analysis in case $\o=1$, we get a finite dimensional linear algebra for $\o=2$.  Despite the lack of practicality, this method have already allowed to better understand the structure of the moments of the Minkowski question mark function. This technique provides information on the eigenvalues $\lambda_{n}$ as well. Possibly, it can be also applied to the study of the eigenvalues of the Laplace-Beltrami operator $\Delta=-y^{2}(\frac{\partial^{2}}{\partial x^2}+\frac{\partial^{2}}{\partial y^2})$.\\

Now we will describe the ``dual" result. Let $s>1$ be a real number. The Mayer-Ruelle operator is defined by \cite{mayer4}
\begin{eqnarray}
\mathcal{L}_{s}[f(t)](x)=\sum\limits_{m=1}^{\infty}\frac{1}{(x+m)^s}f\Big{(}\frac{1}{x+m}\Big{)},\quad f\in\mathbf{V}.
\label{mr}
\end{eqnarray}
It is known \cite{hensley} that the dominant eigenvalue $\lambda_{1}(s)$ is positive and simple. (Note that various notations for $\lambda_{1}(s)$ are in use. For example, $\lambda_{1}(s)=\beta_{1}(s/2)$ in \cite{zagier1}, and $\lambda(s)=\lambda_{1}(s)$ in \cite{hensley}).
The function $s\mapsto\lambda_{1}(s)$ is analytic and strictly decreasing, log concave, it has a meromorphic continuation to the whole complex plane, and for $s\in\mathbb{R}$,
\begin{eqnarray*}
 \lim\limits_{s\rightarrow1_{+}}(s-1)\lambda_{1}(s)=1,\quad \lim\limits_{s\rightarrow\infty}\frac{1}{s}\log \lambda_{1}(s)=-\log \phi.
 \end{eqnarray*}
 This function can be given an alternative definition. For any $k-$tuple of positive integers $\mathbf{a}=(a_{1},a_{2},\ldots,a_{k})$, let $\langle\mathbf{a}\rangle$ denote the denominator of the continued fraction $[0,a_{1},a_{2},\ldots,a_{k}]$, and let $A(k)$ be the set of all such $k-$tuples. Then for $s>1$,
\begin{eqnarray*}
\lambda_{1}(s)=\lim\limits_{k\rightarrow\infty}\Big{(}\sum\limits_{\mathbf{a}\in A(k)}\langle\mathbf{a}\rangle^{-s}\Big{)}^{\frac{1}{k}}.
\end{eqnarray*}
This is the reason why $\lambda_{1}(s)$ is called \emph{the pseudo-zeta function associated with continued fractions}. More of its properties can be found in \cite{hensley}. We note that, thanks to the works of Mayer \cite{mayer3}, Lewis, and Lewis and Zagier \cite{zagier2}, there is a deep connection between spectrum of the Laplace-Beltrami operator, Lewis's three-term functional equation, spectrum of the Mayer-Ruelle operator $\mathcal{L}_{2s}$, and the zeros of the Selberg zeta function. Thus, according to the calculations in \cite{zagier1},
\begin{eqnarray*}
\lambda_{1}(1+2it_{0})&=&-1,\quad t_{0}=9.533695_{+},\text{ corresponding to the first odd spectral parameter};\\
\lambda_{1}(1+2it_{1})&=&1,\quad t_{1}=13.77975_{+},\text{ corresponding to the first even spectral parameter}.
\end{eqnarray*}
Finally, since ${\rm det}(1-\mathcal{L}_{2s})$ also vanishes for
$2s$ being a non-trivial zero of the Riemann zeta function, it so
happens that $1-\lambda_{1}(s)$ shares some zeros with the Riemann
zeta function as well. For example, this is the case for the first
and the fourth non-trivial zero, ordered by their (positive)
imaginary
part. Further calculations are currently being performed by the author.\\

For integers $p,t,j\geq 0$, let us define rational functions with rational coefficients $B_{p,t}(\mathbf{X},\mathbf{Y})$ and $\Lambda_{j}(\mathbf{X},\mathbf{Y})$ by
\begin{eqnarray*}
\Lambda_{0}(\mathbf{X},\mathbf{Y})=2\mathbf{Y}-2,\quad B_{0,0}(\mathbf{X},\mathbf{Y})=1,
\end{eqnarray*}
and then exactly the same way as we defined $A$ and $\Psi$ using procedure and the recurrence (\ref{calc}), subject to the following convention: replace throughout $A$ with $B$, $\Psi$ with $\Lambda$, and replace the underlined binomial coefficient with a simpler one, given by
\begin{eqnarray*}
\binom{k-i}{t-i-j}.
\end{eqnarray*}
Then we have
\begin{conj}Let $s\in\mathbb{C}$ and $\Re s>1$. Then
\begin{eqnarray}
\fbox{$\displaystyle\frac{1}{\lambda_{1}(s)}=\sum\limits_{j=0}^{\infty}(-1)^{j}\Lambda_{j}(s-1,2^{s-1})$}
\label{main2}
\end{eqnarray}
and the series converges absolutely.
\label{conj3}
\end{conj}
It is rather remarkable that all the eigenvalues of $\mathcal{L}$ can be calculated using almost the same recurrence as the one which is used to calculate the dominant eigenvalue of $\mathcal{L}_{s}$!\\

We finish this section with the following curious corollary. It is known that $\lambda'_{1}(2)=-C$, where $C=\frac{\pi^2}{12\log 2}$ is the so called \emph{Khinchin-L\'{e}vy constant} \cite{flajolet3,hensley}. Based on the above series, we can calculate this constant using the truncation of the right hand side of (\ref{main2}). Thus, we have:
\begin{eqnarray*}
\frac{\d}{\d s}\sum\limits_{j=0}^{3}(-1)^{j}\Lambda_{j}(s-1,2^{s-1})|_{s=2}=\frac{167}{48}\cdot\log 2-\frac{59}{48}=\frac{\pi^2}{12\log 2}-0.00416121139_{+}.
\end{eqnarray*}
Further, the summations over $0\leq j\leq 5$ and $0\leq j\leq 8$ produce respectively the following approximations:
\begin{eqnarray*}
\frac{44501}{11520}\cdot\log 2-\frac{1909}{1280}&=&\frac{\pi^2}{12\log 2}-0.00039353038_{+};\\
\frac{66655217}{15482880}\cdot\log 2-\frac{59637137}{33177600}&=&\frac{\pi^2}{12\log 2}-0.00001907581_{+}.\\
\end{eqnarray*}
Thus, each time we get a linear form $a_{N}\log^{2}2+b_{N}\log 2$, $a_{N},b_{N}\in\mathbb{Q}$, which approximates the constant $\pi^{2}$. Though, of course, the approximations are far from being best, it is still a rather curious corollary!\\
\indent The series (\ref{main2}) can be also applied to calculate \emph{the Hensley constant}, given by \cite{finch, hensley}
\begin{eqnarray*}
\frac{\lambda'_{1}(2)^{2}-\lambda''_{1}(2)}{\pi^{6}\lambda'_{1}(2)^{3}}.
\end{eqnarray*}
\section{Heuristics behind the procedure}
In this section we present more or less formal derivation of the recurrence (\ref{calc}). As the Table 2 suggests, all subsequent statements must have rigorous proofs. The main ideas behind our method are the same as in \cite{alkauskas}. \\

Let $\o\in\mathbb{C}$, $|\o-2|\leq 1$, and let us consider the following operator
\begin{eqnarray*}
\mathcal{D}_{\o}[f(t)](z)=\sum\limits_{m=1}^{\infty}\frac{\o^{m-1}}{\Big{(}\o^{m-1}z-\frac{\o^{m}-1}{\o-1}\Big{)}^2}\cdot
f\Big{(}\frac{\o}{\o^{m-1}z-\frac{\o^{m}-1}{\o-1}}\Big{)}.
\end{eqnarray*}
 Up to the sign, this operator is the generalization of (\ref{gkw}). Indeed,
\begin{eqnarray*}
\mathcal{D}_{1}[f(-t)](-z)=\mathcal{L}[f(t)](z).
\end{eqnarray*}
In these formal derivations we do not specify where this operator is defined. If $\o$ is real, then
\begin{eqnarray*}
[-\o,0]=\bigcup\limits_{m=1}^{\infty}\Big{[}-\frac{\o(\o-1)}{\o^{m}-1},-\frac{\o(\o-1)}{\o^{m+1}-1}\Big{)}\cup\{0\}.
\end{eqnarray*}
Thus, in this case $\mathcal{D}_{\o}$ can be defined as an operator in the space ${\sf C}[-\o,0]$. For complex $\o$, the interval $[-\o,0]$ is replaced by a closure of a direct limit of $z=0$ under the family of transformations $T_{m}(\o,z)=\frac{\o}{\o^{m-1}z-\frac{\o^{m}-1}{\o-1}}$, $m\in\mathbb{N}$. This is a Julia set $\mathscr{H}_{\o}$ similar to the one introduced in \cite{alkauskas}, and we know that its structure can be rather complicated and it deserves a separate treatment. Thus, $\mathcal{D}_{\o}$ acts on the space ${\sf C}(\mathscr{H}_{\o})$. Of course, we could have chosen the definition of $\mathcal{D}_{\o}$ in such a way that $\mathcal{D}_{1}=\mathcal{L}$, since this is only a matter of orientation of the plane. Nevertheless, to avoid confusion, it is better to adhere to the same orientation as in \cite{alkauskas}.
If $G(\o,z)$ is the eigenfunction of this operator with the eigenvalue $\frac{1}{\lambda(\o)}$, then it satisfies the functional equation
\begin{eqnarray*}
G(\o,z+1)=\o G(\o,\o z)+\frac{\lambda(\o)}{z^{2}}\cdot G\Big{(}\o,\frac{\o}{z}\Big{)}.
\end{eqnarray*}
Our result in Section 1 is based on the following

\begin{assum}For each $n\in\mathbb{N}$, there exist a two variable analytic function $G_{n}(\o,z)$, which is defined for $(\o,z)\in\mathbb{C}^{2}$, $|\o-2|\leq 1$, $z\in\mathbb{C}\setminus\{\mathscr{I}_{\o}+1\}$, and the unique analytic function $\lambda_{n}(\o)$, $|\o-2|\leq 1$, such that $\lim\limits_{z} G(\o,z)=0$ if ${\rm dist}(z,\mathscr{I}_{\o})\rightarrow\infty$, $G_{n}(2,z)$ is a meromorphic function with a single pole of order $n+1$ at $z=2$, and $G(\o,z)$ satisfies the functional equation
\begin{eqnarray}
G(\o,z+1)=\o G(\o,\o z)+\frac{(-1)^{n+1}\lambda_{n}(\o)}{z^{2}}\cdot G\Big{(}\o,\frac{\o}{z}\Big{)},\quad |\o-2|\leq 1, z\notin\mathscr{I}_{\o}.
\label{base}
\end{eqnarray}
\label{priel}
\end{assum}
\noindent(Beware that for $n=1$ the function $\lambda_{1}(\o)$ is not the same as the function $\lambda_{1}(s)$, as is given in the second half of the Section 1. Here we deal only with $\mathcal{L}$, and this should not cause any confusion). The set $\mathscr{I}_{\o}$ was introduced in \cite{alkauskas} and it is a Julia set. More precisely, it is a closure of a direct limit of iterated function system, given by the two transformations (\ref{itf}). The set $\mathscr{I}_{\o}$ is homeomorphic to a unit interval for $\o\neq 1$. If $1\leq \o\leq 2$ is real, then $\mathscr{I}_{\o}=[\o-1,\frac{1}{\o-1}]$. For complex $\o$, the structure of $\mathscr{I}_{\o}$ is rather complicated. We refer to \cite{alkauskas} for further details.
When $\o=1$, the equation (\ref{base}) specializes to
\begin{eqnarray*}
G(1,z+1)=G(1,z)+\frac{(-1)^{n+1}\lambda_{n}(1)}{z^2}\cdot G\Big{(}1,\frac{1}{z}\Big{)}.
\end{eqnarray*}
Comparing this to (\ref{vienas}), we claim that
\begin{eqnarray}
G_{n}(1,-z)=c_{n}\Phi'_{n}(z),\quad c_{n}\neq 0,\quad (-1)^{n+1}\lambda_{n}(1)=\frac{1}{\lambda_{n}}.\label{tikr}
\end{eqnarray}

As is stated in the Assumption \ref{priel}, for $\o=2$ the equation (\ref{base}) has a solution of the form
\begin{eqnarray*}
G_{n}(2,z)=G_{n}(z)=\frac{P_{n}(z)}{(z-2)^{n+1}},\text{ where }P_{n}(z)\text{ is entire},\quad P_{n}(2)\neq 0.
\end{eqnarray*}
In terms of $P_{n}(z)$, the equation (\ref{base}) then reads as
\begin{eqnarray}
P_{n}(z+1)=\frac{1}{2^{n}}\cdot P_{n}(2z)+\frac{\lambda_{n}(2)}{2^{n+1}}\cdot P_{n}\Big{(}\frac{2}{z}\Big{)}z^{n-1}.
\label{polin}
\end{eqnarray}
A substitution $z=1$ yields
\begin{eqnarray*}
\lambda_{n}(2)=2^{n+1}-2.
\end{eqnarray*}
Without a loss of generality, let us assume $P_{n}(2)=1$. If we differentiate (\ref{polin}) and substitute $z=1$, we immediately obtain
\begin{eqnarray*}
P'_{n}(2)=\frac{\d}{\d z}P_{n}(z)|_{z=2}=\frac{2^{n}-1}{3\cdot2^{n}-4}\cdot(n-1).
\end{eqnarray*}
Further differentiations of (\ref{polin}) and subsequent substitutions $z=1$ output the values
\begin{eqnarray*}
\frac{\d^2}{\d z^{2}}P_{n}(2)&=&\frac{2^{n}-1}{3(3\cdot2^{n}-4)}\cdot(n-1)(n-2);\\
\frac{\d^3}{\d z^{3}}P_{n}(2)&=&\frac{(2^n-2)(2^n-1)}{(3\cdot 2^n-4)(9\cdot 2^n-16)}\cdot(n-1)(n-2)(n-3).
\end{eqnarray*}
In general, a direct induction shows that
\begin{eqnarray*}
\frac{\d^{\ell}}{\d z^{\ell}}P_{n}(2)=\frac{(n-1)!}{(n-\ell-1)!}\cdot c^{(\ell)}_{n},
\end{eqnarray*}
where $c^{(\ell)}_{n}$ are given recurrently
\begin{eqnarray*}
c^{(0)}_{n}=1,\quad c^{(\ell)}_{n}=\frac{2^n-1}{2^n\cdot A_{\ell}}\cdot\sum\limits_{i=0}^{\ell-1}c^{(i)}_{n}(-1)^{i}2^{i}\binom{\ell}{i};
\end{eqnarray*}
here $A_{\ell}=1-2^{\ell}$ if $\ell$ is even, and $A_{\ell}=2^{\ell}+1-2^{\ell+1-n}$ is $\ell$ is odd. Thus, we see that not only $P_{n}(z)$ is an entire function but it is a polynomial of degree $\leq n-1$. Moreover, this also shows that the eigenvalue $\lambda_{n}(2)$ is simple.\\

As the second step, let us differentiate (\ref{base}) with respect to $\o$ and substitute $\o=2$. For simplicity, put
\begin{eqnarray*}
\frac{\d}{\d \o}G_{n}(\o,z)|_{\o=2}=H_{n}(z),\quad \lambda'_{n}(2)=\sigma_{n}(2).
\end{eqnarray*}
We get:
\begin{eqnarray}
& &H_{n}(z+1)-2H_{n}(2z)-\frac{(-1)^{n+1}\lambda_{n}(2)}{z^{2}}\cdot H_{n}\Big{(}\frac{2}{z}\Big{)}\nonumber\\
&=&G_{n}(2z)+2z\cdot G'_{n}(2z)+\frac{(-1)^{n+1}\sigma_{n}(2)}{z^{2}}\cdot G_{n}\Big{(}\frac{2}{z}\Big{)}+\frac{(-1)^{n+1}\lambda_{n}(2)}{z^{3}}\cdot G'_{n}\Big{(}\frac{2}{z}\Big{)}.
\label{isv}
\end{eqnarray}
Similarly as in \cite{alkauskas}, we see that
\begin{eqnarray*}
H_{n}(z)=\frac{Q_{n}(z)}{(z-2)^{n+2}},\quad Q_{n}(z)\text{ is polynomial},{\rm deg}(Q_{n})\leq n.
\end{eqnarray*}
Writing the equation (\ref{isv}) in terms of $Q_{n}(z)$, and substituting $z=1$, we obtain
\begin{eqnarray*}
Q_{n}(2)=\frac{2^{n}-2}{3\cdot 2^{n}-2}\cdot(n+1).
\end{eqnarray*}
Differentiation of that equation and substitution $z=1$ (incidentally, the coefficient at $Q'_{n}(2)$ vanishes) gives the unique value for $\sigma_{n}(2)$. We clearly see that, unlike $Q_{n}(2)$, the value of $Q'_{n}(2)$ depends on our choice. This has the following explanation. If $G_{n}(\o,z)$ is a solution to (\ref{base}), the so is $G_{n}(\o,z)\cdot\Theta(\o)$ for any analytic function $\Theta(\o)$, $|\o-2|\leq 1$. In other words: the solution to (\ref{isv}) is not unique; all solutions are given by $H_{n}(z)=\overline{H}_{n}(z)+c\cdot G_{n}(z)$, $c\in\mathbb{C}$, where $\overline{H}_{n}(z)$ is any particular solution. Therefore, $Q_{n}(z)=\overline{Q}_{n}(z)+c\cdot(z-2)P_{n}(z)$, and $Q'_{n}(2)$ is not uniquely defined. Since $P_{n}(2)\neq 0$, we can define it arbitrarily. All the consequent values of $\frac{\d^\ell}{\d z^{\ell}}Q_{n}(2)$, $\ell\geq 2$, are then determined. At this stage, the way our method works should be clear: we differentiate the equation (\ref{base}) $p$ times with respect to $\o$, and substitute $\o=2$. Let $\mathscr{E}_{p}$ be the equation we obtain,  and let
\begin{eqnarray*}
\frac{1}{p!}\frac{\d^{p}}{\d \o^{p}}G_{n}(\o,z)|_{\o=2}=G^{(p)}_{n}(z).
\end{eqnarray*}
The careful inspection of the equation $\mathscr{E}_{p}$ shows that it has the form similar to (\ref{isv}), where on the l.h.s. we have $G^{(p)}_{n}(z)$ instead of $H_{n}(z)$, and the r.h.s. is a $\mathbb{Q}(z)-$linear combination of $G^{(k)}_{n}(z)$ (with arguments $z+1$, $2z$ and $\frac{2}{z}$) for $0\leq k\leq p-1$, and it includes the coefficients $\frac{\d^{k}}{\d\o^{k}}\lambda_{n}(2)$, $0\leq k\leq p$. This shows that
\begin{eqnarray*}
G^{(p)}_{n}(z)=\frac{P_{n,p}(z)}{(z-2)^{n+p+1}},{\rm deg}(P_{n,p})\leq n+p-1.
\end{eqnarray*}
The equation $\mathscr{E}_{p}$, written in the terms of $P_{n,p}(z)$, allows to determine $P_{n,p}(2)$. We inductively differentiate the equation $\mathscr{E}_{p}$ exactly $t$ times, $0\leq t<p$, substitute $z=1$ to recover the value $\frac{\d^{t}}{\d z^{t}}P_{n,p}(2)$. When $t=p$, the coefficient at $\frac{\d^{p}}{\d z^{p}}P_{n,p}(2)$ vanishes, but then $\frac{\d^{p}}{\d z^{p}}\lambda_{n}(2)$ appears on the stage, which thus can be calculated. We set $\frac{\d^{p}}{\d z^{p}}P_{n,p}(2)=0$, and further proceed with determination of $\frac{\d^{t}}{\d z^{t}}P_{n,p}(2)$ for $t>p$. Having done so for $0\leq t\leq N$, let us move to the equation $\mathscr{E}_{p+1}$, and so on. To write down the equations $\mathscr{E}_{p}$ is a tedious job, and we rather proceed via the formal power series expansion.\\

As we have just seen, the $n$th solution of the functional equation (\ref{base}) can be given by
\begin{eqnarray*}
G_{n}(\o,z)&=&\sum\limits_{k=0}^{\infty}\frac{P_{n,k}(z)}{(z-2)^{n+k+1}}(\o-2)^{k}=
\sum\limits_{k=0}^{\infty}\sum\limits_{i=0}^{n+k-1}\frac{a_{k,i}(n)}{(z-2)^{n+k+1-i}}(\o-2)^{k};\\
\text{ where }P_{n,k}(z)&=&\sum\limits_{i=0}^{n+k-1}a_{k,i}(n)(z-2)^{i}.
\end{eqnarray*}
Let
\begin{eqnarray}
\lambda_{n}(\o)=\sum\limits_{j=0}^{\infty}\psi_{j}(n)(\o-2)^{j}.
\label{tikr2}
\end{eqnarray}
As we already know, $\psi_{0}(n)=2^{n+1}-2$. We will formally substitute these powers series into (\ref{base}) and compare the corresponding coefficients at $\frac{(\o-2)^{p}}{(z-1)^{n+p+1-t}}$. First, we need two formulas.
\begin{lem}Let $A\in\mathbb{N}_{0}$. Then we have the formal power series
\begin{eqnarray}
\frac{1}{(\o z-2)^{A+1}}&=&\sum\limits_{s=0}^{\infty}\sum\limits_{\ell=0}^{s}\binom{A+s}{s}\binom{s}{\ell}\frac{(-1)^{s}}{2^{A+s+1}}
\frac{(\o-2)^{s}}{(z-1)^{A+\ell+1}};\label{form1}\\
\frac{1}{z^{2}(\frac{\o}{z}-2)^{A+1}}&=&\sum\limits_{s=0}^{\infty}\sum\limits_{\ell=0}^{A-1}\binom{A+s}{s}\binom{A-1}{\ell}\frac{(-1)^{A+1}}{2^{A+s+1}}
\frac{(\o-2)^{s}}{(z-1)^{A+s+1-\ell}}\label{form2}.
\end{eqnarray}
\end{lem}
\begin{proof}Let us write
\begin{eqnarray*}
\frac{1}{(\o z-2)^{A+1}}=\sum\limits_{s=0}^{\infty}(\o-2)^{s}T_{s}(z).
\end{eqnarray*}
Differentiation $s$ times with respect to $\o$ and substitution $\o=2$ implies
\begin{eqnarray*}
T_{s}(z)=\frac{(-1)^{s}z^{s}}{2^{A+s+1}(z-1)^{A+s+1}}\binom{A+s}{s}.
\end{eqnarray*}
Now the formula (\ref{form1}) follows from the identity $z^s=(z-1+1)^{s}=\sum_{\ell=0}^{s}\binom{s}{l}(z-1)^{s-\ell}$. The same trick proves the formula (\ref{form2}).\end{proof}
Therefore, formally, we have:
\begin{eqnarray*}
G_{n}(\o,z+1)=\sum\limits_{k=0}^{\infty}\sum\limits_{i=0}^{n+k-1}a_{k,i}(n)\frac{(\o-2)^{k}}{{(z-1)^{n+k+1-i}}}.
\end{eqnarray*}
Further, using (\ref{form1}), we get
\begin{eqnarray*}
& &\o G_{n}(\o,\o z)=\sum\limits_{k=0}^{\infty}\sum\limits_{i=0}^{n+k-1}\frac{a_{k,i}(n)}{(\o z-2)^{n+k+1-i}}\o(\o-2)^{k}\\
&=&\sum\limits_{k=0}^{\infty}\sum\limits_{i=0}^{n+k-1}\sum\limits_{s=0}^{\infty}\sum\limits_{\ell=0}^{s}a_{k,i}(n)
\binom{n+k+s-i}{s}\binom{s}{\ell}\frac{(-1)^{s}}{2^{n+k+1+s-i}}
\frac{\o(\o-2)^{s+k}}{(z-1)^{n+k+1+\ell-i}}.\\
\end{eqnarray*}
To get the expansion in terms of powers of $\o-2$, we will soon use the identity $\o(\o-2)^{s+k}=(\o-2)^{s+k+1}+2(\o-2)^{s+k}$.
Finally, using (\ref{form2}), we obtain
\begin{eqnarray*}
&(-1)^{n+1}&\frac{\lambda_{n}(\o)}{z^{2}}\cdot G_{n}\Big{(}\o,\frac{\o}{z}\Big{)}=
(-1)^{n+1}\sum\limits_{j=0}^{\infty}\sum\limits_{k=0}^{\infty}\sum\limits_{i=0}^{n+k-1}
\frac{\psi_{j}(n)\cdot a_{k,i}(n)}{z^{2}(\frac{\o}{z}-2)^{n+k+1-i}}(\o-2)^{j+k}\\
&=&\sum\limits_{j=0}^{\infty}\sum\limits_{k=0}^{\infty}\sum\limits_{i=0}^{n+k-1}\sum\limits_{s=0}^{\infty}\sum\limits_{\ell=0}^{n+k-i-1}
\psi_{j}(n)\cdot a_{k,i}(n)\binom{n+k+s-i}{s}\binom{n+k-i-1}{\ell}\\
&\times&\frac{(-1)^{k-i}}{2^{n+k+1+s-i}}
\frac{(\o-2)^{s+j+k}}{(z-1)^{n+k+s+1-i-\ell}}.
\end{eqnarray*}
Since $G_{n}(\o,z)$ satisfies (\ref{base}), let us compare the coefficient at the term $\frac{(\o-2)^{p}}{(z-1)^{n+p+1-t}}$ of the power series expansions of both sides. We get
\begin{eqnarray*}
a_{p,t}(n)&=&\sum\limits_{j=0}^{p}\sum\limits_{k=0}^{p-j}\sum\limits_{i=0}^{t-j}
\psi_{j}(n)\cdot a_{k,i}(n)\binom{n+p-i-j}{p-j-k}\binom{n+k-i-1}{t-i-j}\frac{(-1)^{k-i}}{2^{n+p+1-i-j}}\\
&+&\sum\limits_{k=0}^{p}\sum\limits_{i=0}^{t}a_{k,i}(n)
\binom{n+p-i}{p-k}\binom{p-k}{p+i-k-t}\frac{(-1)^{p-k}}{2^{n+p-i}}\\
&+&\sum\limits_{k=0}^{p-1}\sum\limits_{i=0}^{t-1}a_{k,i}(n)
\binom{n+p-i-1}{p-k-1}\binom{p-k-1}{p+i-k-t}\frac{(-1)^{p-k-1}}{2^{n+p-i}}.
\end{eqnarray*}
This shows that
\begin{eqnarray*}
a_{p,t}(n)=A_{p,t}(n,2^{n}),\quad \psi_{j}(n)=\Psi_{j}(n,2^{n}),
\end{eqnarray*}
where $A_{p,t}$ and $\Psi_{j}$ are rational functions with integer coefficients. After some elementary transformations, we arrive exactly to the recurrence given in the first section. Finally, if we compare (\ref{tikr}) and (\ref{tikr2}), we arrive to the Conjecture \ref{conj2}. \\

Concerning the ``dual" result, we start from the equation
\begin{eqnarray}
\widehat{G}_{s}(\o,z+1)&=&\o \widehat{G}_{s}(\o,\o z)+\frac{\lambda_{1}(\o,s)}{(-z)^{s}}\cdot \widehat{G}_{s}\Big{(}\o,\frac{\o}{z}\Big{)},\label{vien}
\\ |\o-2|&\leq& 1,\quad z\notin\mathscr{I}_{\o},\quad\Re s>1.\nonumber
\end{eqnarray}
$\widehat{G}_{s}(\o,z)$ is a function of three complex variables! We expect that $\widehat{G}_{s}(1,-z)$ is the eigenfunction corresponding to the leading eigenvalue of $\mathcal{L}_{s}$. Note that different branches of a multi-valued function $(-z)^{s}$ differ by factors of the form $e^{2\pi iNs}$. Since we expect that $\widehat{G}_{s}(\o,z)(z-2)^{s}$ is single-valued (see below), the multi-valued nature of $(-z)^{s}$ does not cause complications, if properly defined.  All our subsequent formal calculations are based on the assumption that (up to normalization)
\begin{eqnarray*}
\widehat{G}_{s}(2,z)=\frac{1}{(z-2)^{s}}.\text{ This gives }\lambda_{1}(2,s)=2^{s}-2.
\end{eqnarray*}
For further steps we assume that the solution of (\ref{vien}) is given by
\begin{eqnarray*}
\widehat{G}_{s}(\o,z)=\sum\limits_{k=0}^{\infty}\frac{\widehat{P}_{s,k}(z)}{(z-2)^{s+k}}(\o-2)^{k},\quad
\widehat{P}_{s,k}(z)\text{ is polynomial in } z,\quad{\rm deg}(\widehat{P}_{s,k})\leq k.
\end{eqnarray*}
All the subsequent steps are analogous to the ones which lead to the Conjecture \ref{conj2}. Since the calculations match the known numerical data very well (including the calculations of the Khinchin-L\'{e}vy constant in the end of Section 1), the Conjecture \ref{conj3} should be true.
\appendix
\section{}
In this appendix we list various tables and figures. For the
convenience of referee and other interested readers, the MAPLE code
to compute rational functions $\Psi_{j}(\mathbf{X},\mathbf{Y})$ can
be downloaded from
\url{http://www.alkauskas.puslapiai.lt/MP3/gkw.txt}
\begin{figure}
\centering
\begin{tabular}{c c}

\epsfig{file=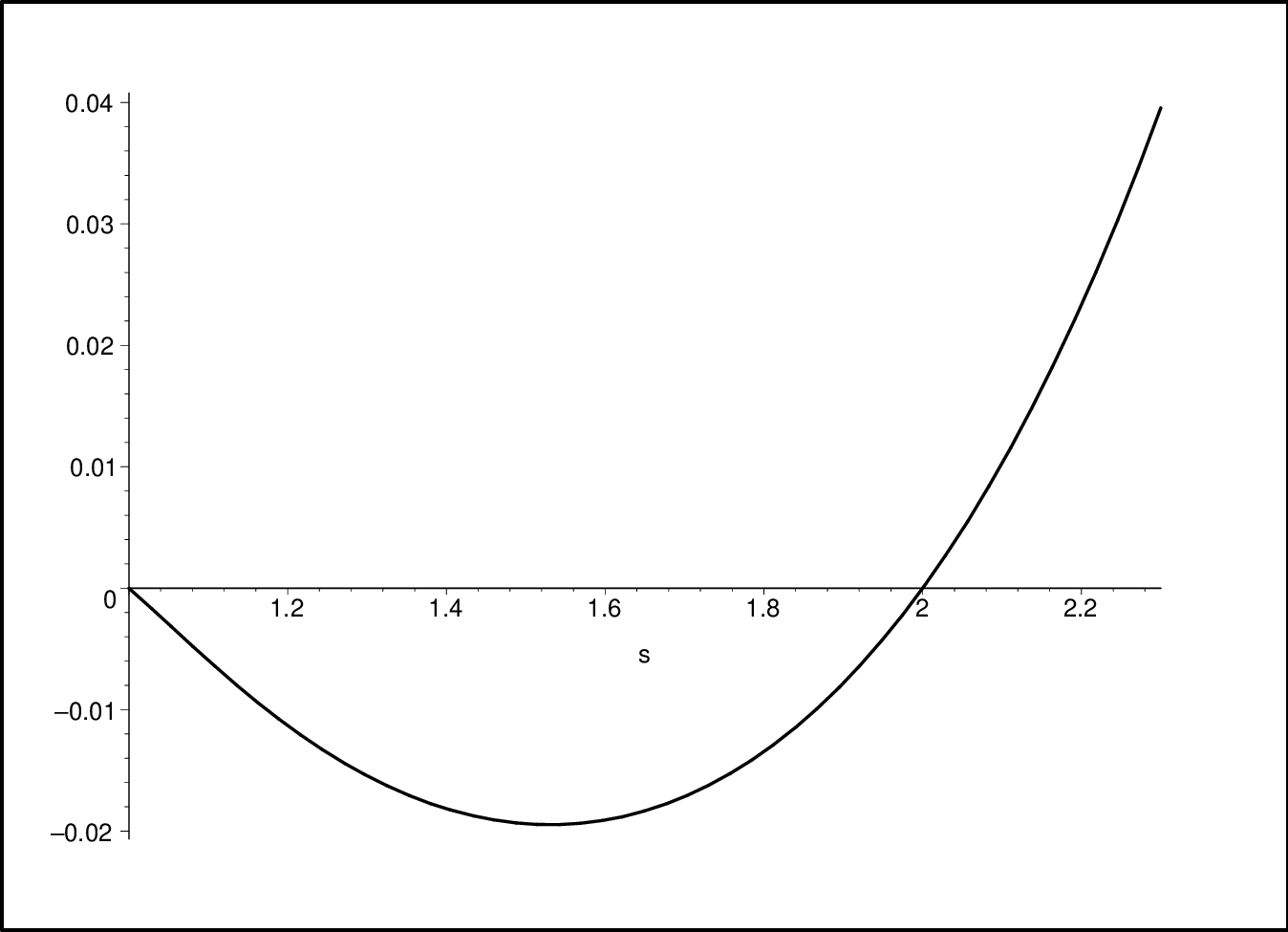,width=180pt,height=215pt,angle=-90}
&\epsfig{file=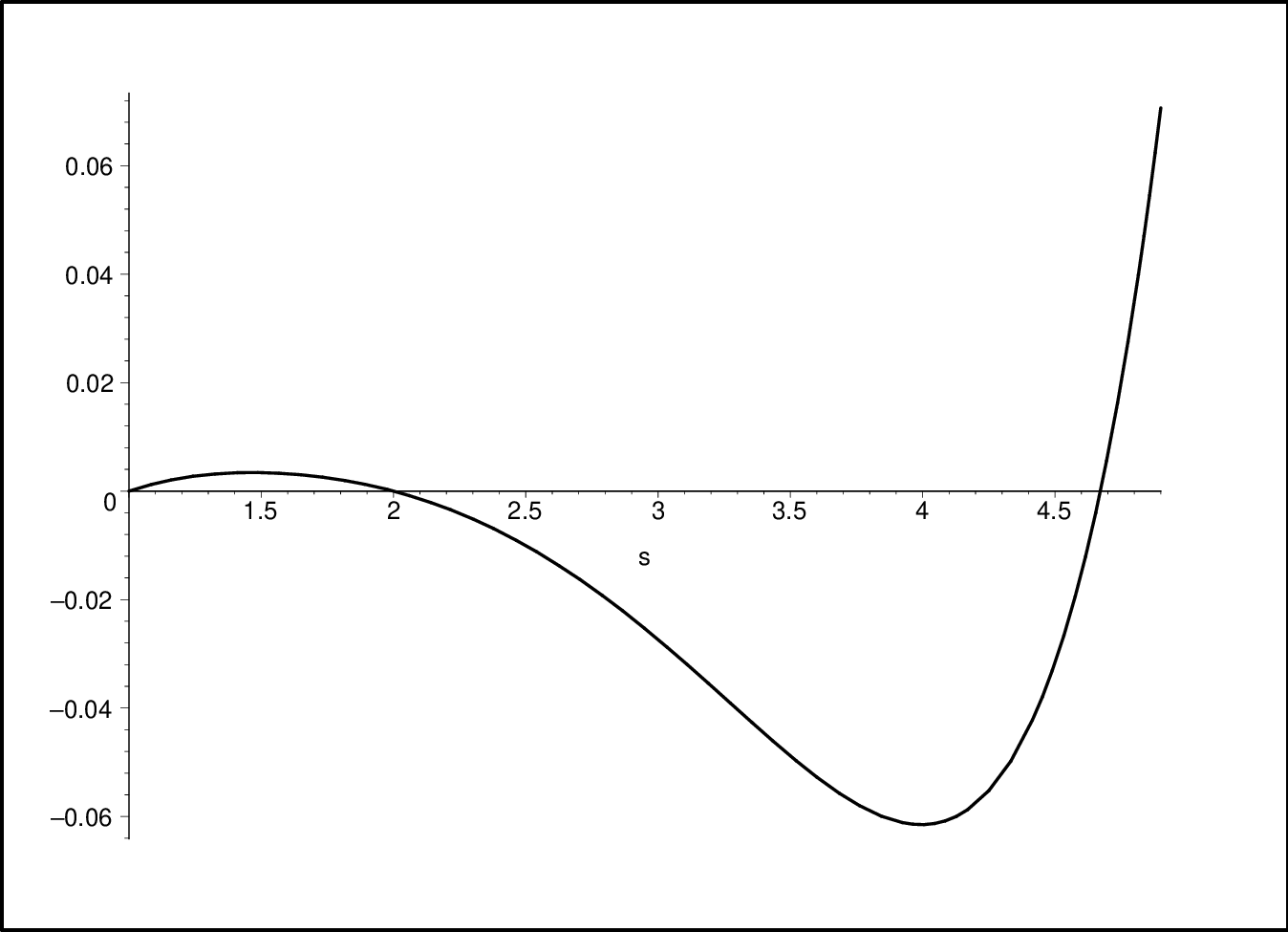,width=180pt,height=215pt,angle=-90} \\
$\Lambda_{2}(s-1,2^{s-1})$, $s\in[1,2.4]$ & $\Lambda_{3}(s-1,2^{s-1})$, $s\in[1,4.7]$\\
\epsfig{file=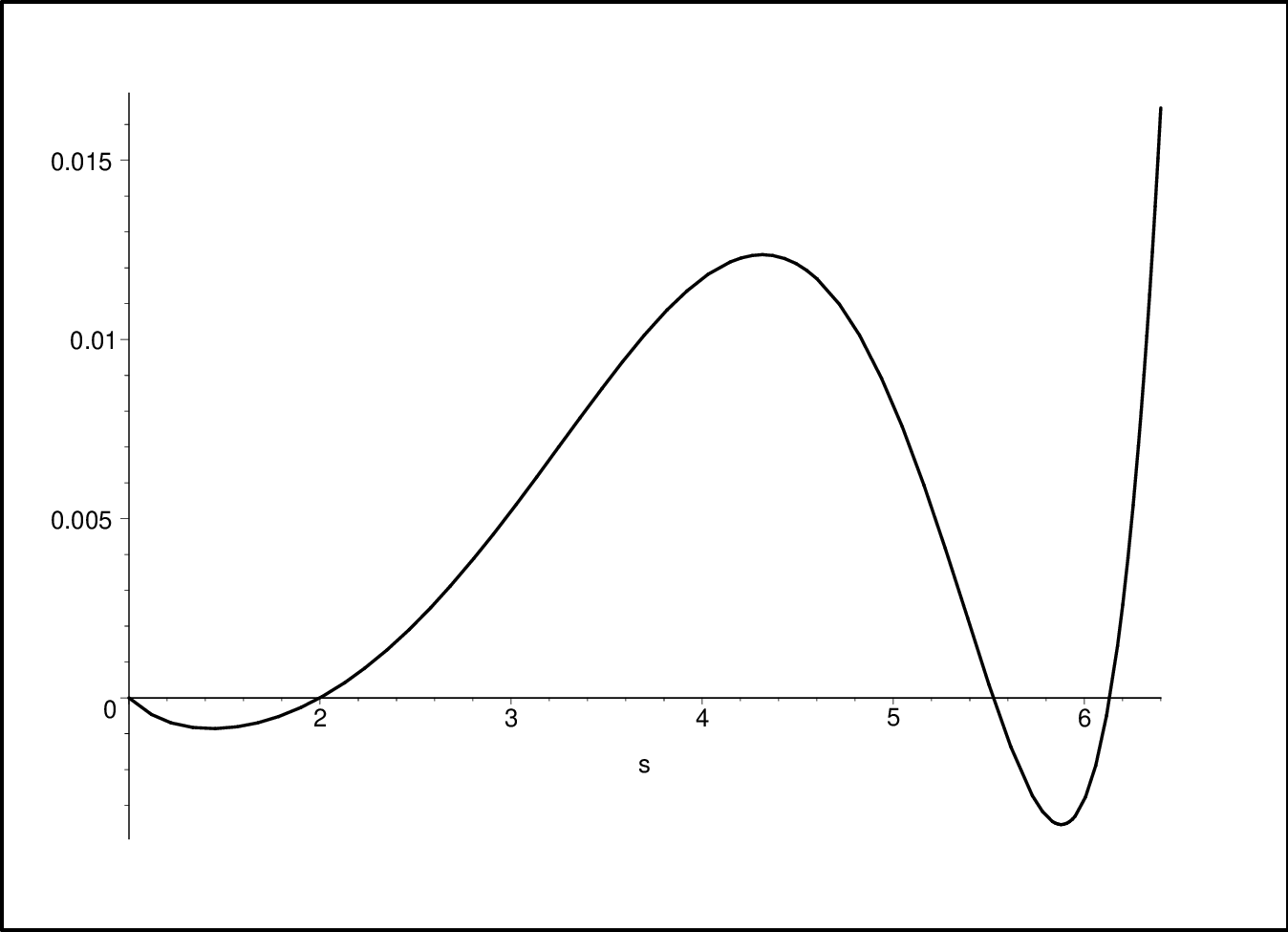,width=180pt,height=215pt,angle=-90} &
\epsfig{file=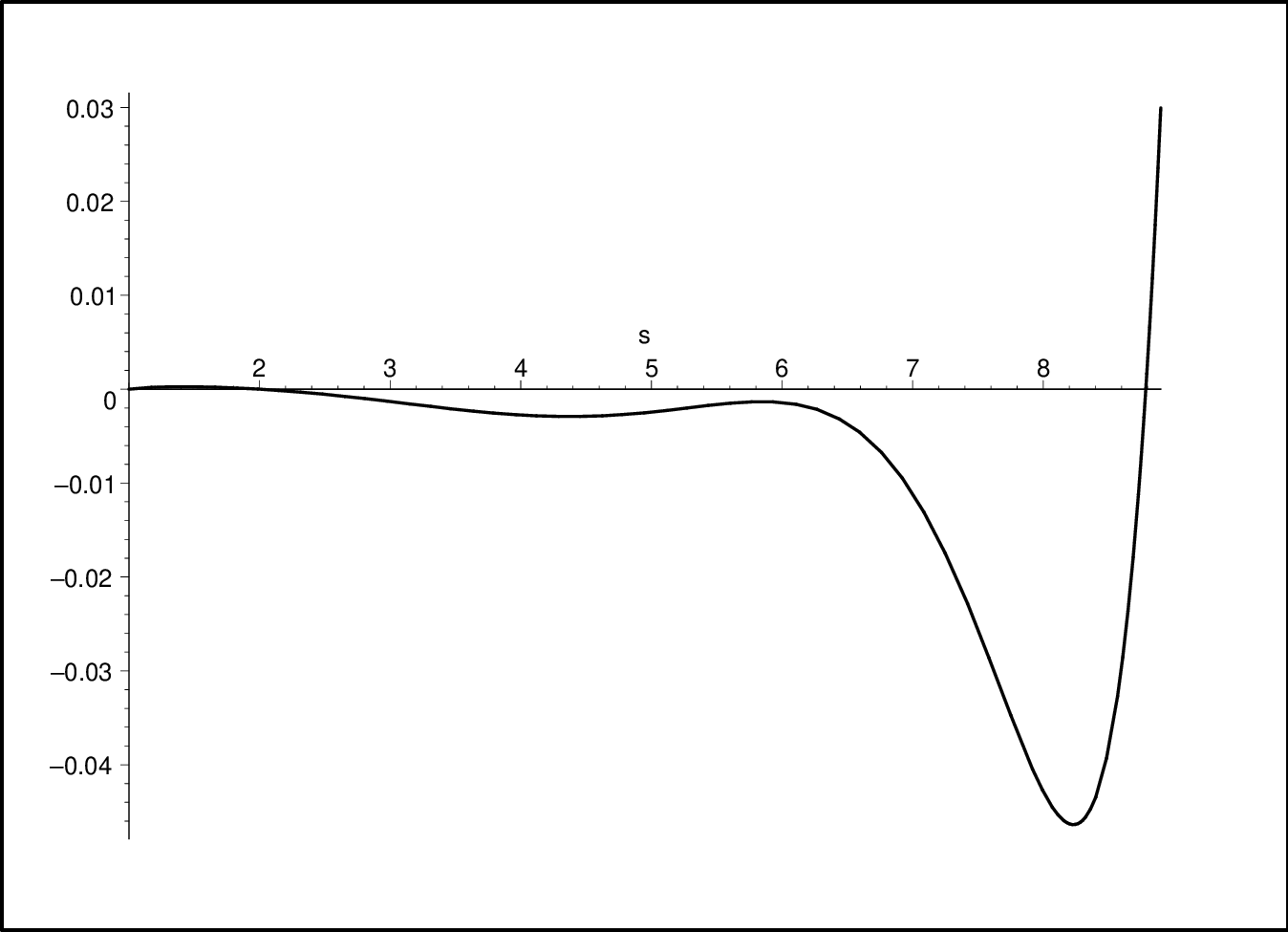,width=180pt,height=215pt,angle=-90}\\
$\Lambda_{4}(s-1,2^{s-1})$, $s\in[1,6.3]$ & $\Lambda_{5}(s-1,2^{s-1})$, $s\in[1,8.9]$\\
\epsfig{file=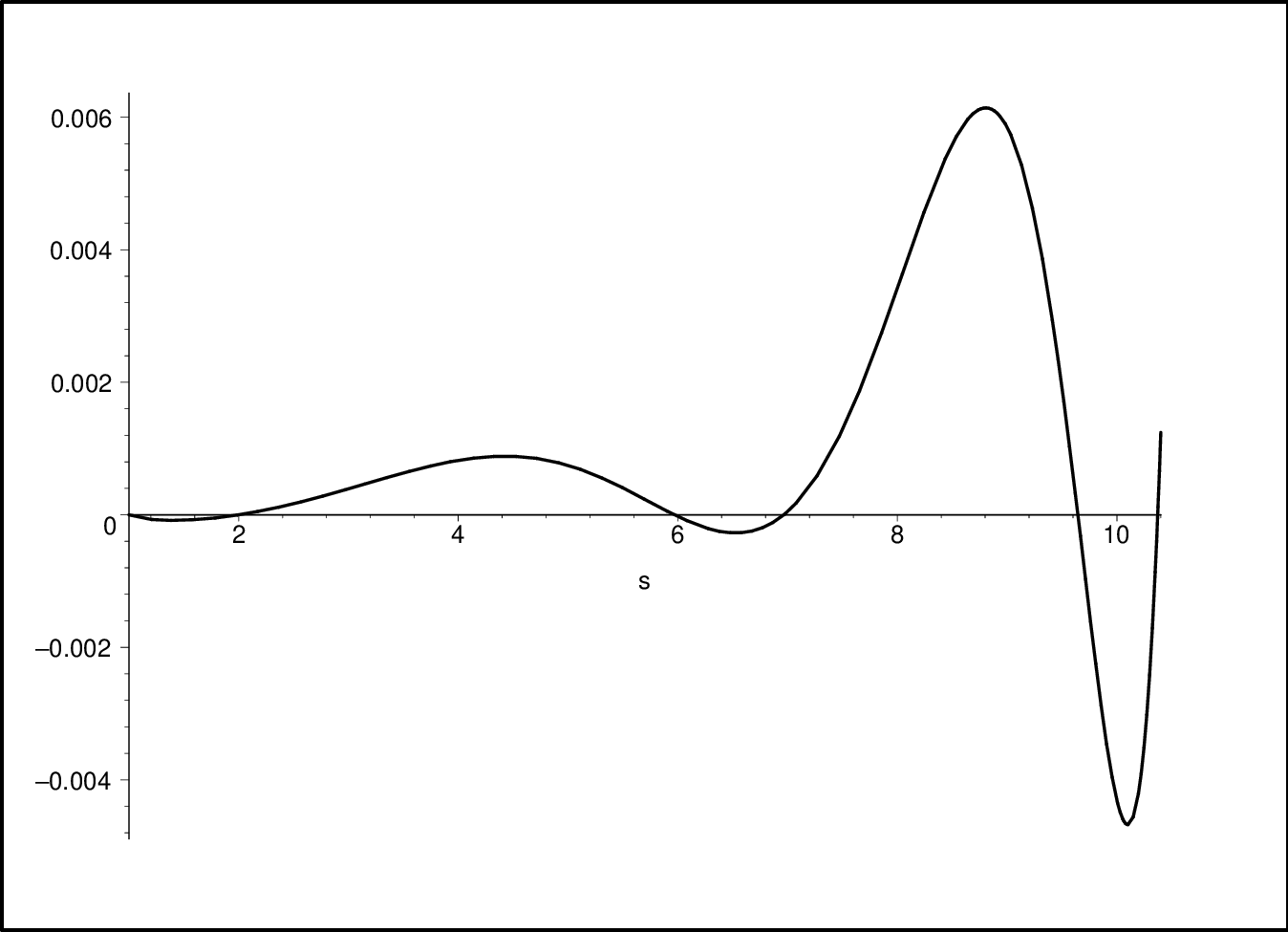,width=180pt,height=215pt,angle=-90} &
\epsfig{file=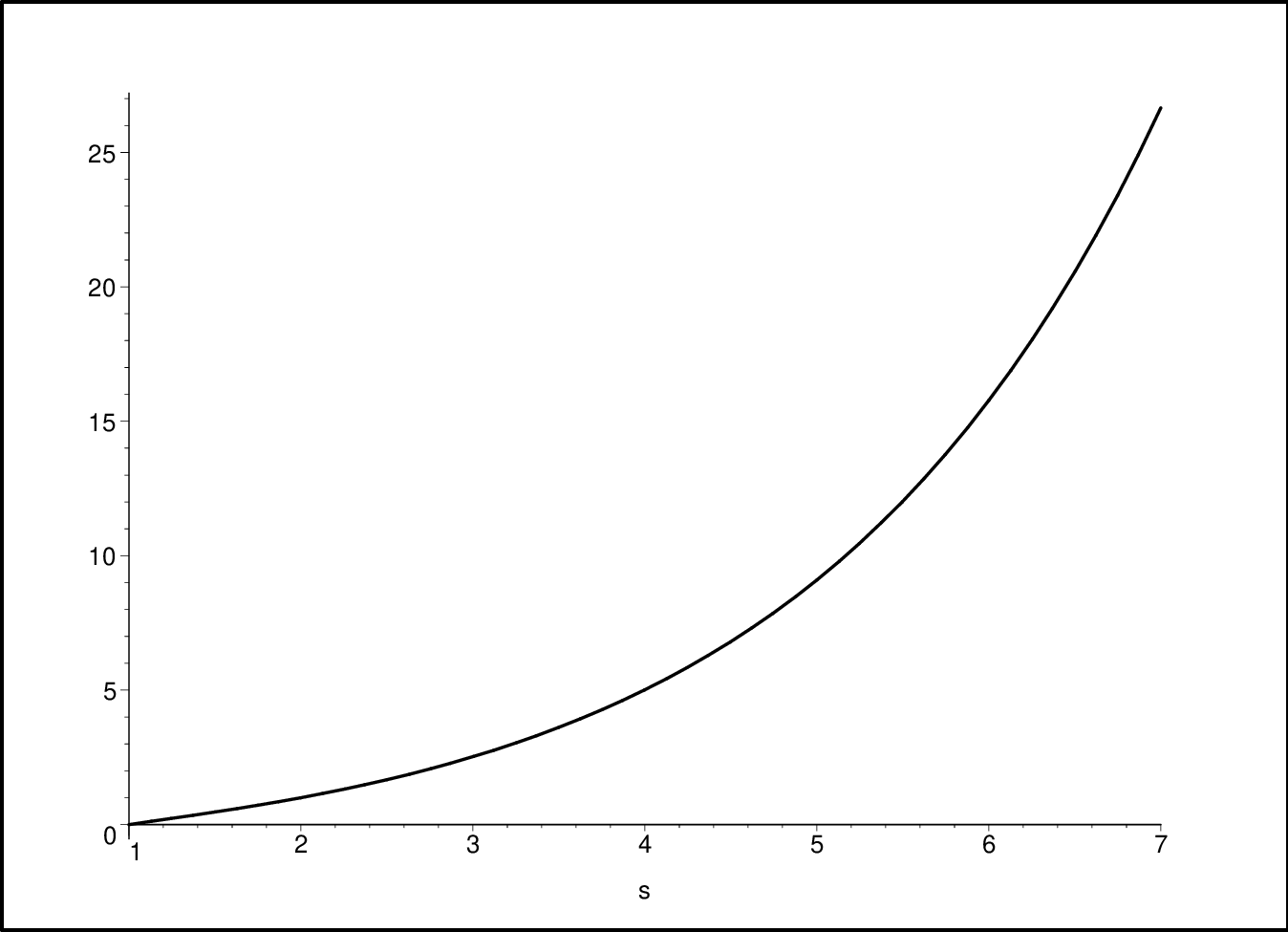,width=180pt,height=215pt,angle=-90}\\
$\Lambda_{6}(s-1,2^{s-1})$, $s\in[1,10.4]$ & $\sum_{j=0}^{8}(-1)^{j}\Lambda_{j}(s-1,2^{s-1})$, $s\in[1,7]$\\
\end{tabular}
\caption{Functions $\Lambda_{j}(s-1,2^{s-1})$}.
\end{figure}
\begin{center}
\noindent\begin{tabular}{||r || c ||} \hline
\multicolumn{2}{||c||}{\textbf{1. Rational functions $\Psi_{j}(\mathbf{X},\mathbf{Y})$}, $0\leq j\leq 3$.}\\
\hline
&\\
$0$&$2Y-2$\\
& \\
$1$&$\displaystyle\frac{-3XY^3+9Y^3+4X^2Y^2+6XY^2-27Y^2-4X^2Y-4XY+26Y-8}{(3Y-4)(3Y-2)}$\\
&  \\
$2$&$\scriptstyle(243X^2Y^7-405XY^7-216X^4Y^6+54X^2Y^6+2430XY^6+648X^4Y^5-216X^3Y^5-4212X^2Y^5-6804XY^5$\\
$ $&$\scriptstyle-720X^4Y^4+1584X^3Y^4+11400X^2Y^4+11016XY^4+288X^4Y^3-3008X^3Y^3-13248X^2Y^3-10336XY^3$\\
$ $&$\scriptstyle+2176X^3Y^2+7296X^2Y^2+5120XY^2-512X^3Y-1536X^2Y-1024XY)\times(12(3Y-4)^3(3Y-2)^3)^{-1}$\\
& \\
$3$&$\scriptstyle(-531441X^3Y^{13}-2125764X^2Y^{13}+4074381XY^{13}$\\
$ $&$\scriptstyle+209952X^6Y^{12}-2466936X^4Y^{12}+6377292X^3Y^{12}+15011568X^2Y^{12}-48892572XY^{12}$\\
$ $&$\scriptstyle-419904X^6Y^{11}+3726648X^5Y^{11}+18869436X^4Y^{11}+37747620X^3Y^{11}-24958044X^2Y^{11}+256710060XY^{11}$\\
$ $&$\scriptstyle-3779136X^6Y^{10}-33802272X^5Y^{10}-47530800X^4Y^{10}+161593056X^3Y^{10}-62997264X^2Y^{10}-774372960XY^{10}$\\
$ $&$\scriptstyle+22371552X^6Y^9+120885696X^5Y^9-16282944X^4Y^9-568577232X^3Y^9+260620416X^2Y^9+1465717680XY^9$\\
$ $&$\scriptstyle-57583872X^6Y^8-215488512X^5Y^8+406171584X^4Y^8+1568232576X^3Y^8-106759296X^2Y^8-1745883072XY^8$\\
$ $&$\scriptstyle+86686848X^6Y^7+175893120X^5Y^7-1202361408X^4Y^7-3156231744X^3Y^7-984602304X^2Y^7+1145164608XY^7$\\
$ $&$\scriptstyle-80727552X^6Y^6+25159680X^5Y^6+1970357760X^4Y^6+4467896064X^3Y^6+2532955392X^2Y^6-46116864XY^6$\\
$ $&$\scriptstyle+45333504X^6Y^5-202844160X^5Y^5-2057522688X^4Y^5-4360743936X^3Y^5-3125389824X^2Y^5-661966848XY^5$\\
$ $&$\scriptstyle-12926976X^6Y^4+209399808X^5Y^4+1402464256X^4Y^4+2862624768X^3Y^4+2283358208X^2Y^4+644493312XY^4$\\
$ $&$\scriptstyle-147456X^6Y^3-110641152X^5Y^3-607199232X^4Y^3-1205420032X^3Y^3-1008033792X^2Y^3-308854784XY^3$\\
$ $&$\scriptstyle+1245184X^6Y^2+31653888X^5Y^2+152535040X^4Y^2+294256640X^3Y^2+249659392X^2Y^2+78315520XY^2$\\
$ $&$\scriptstyle-262144X^6Y-3932160X^5Y-17039360X^4Y-31719424X^3Y-26738688X^2Y-8388608XY)$\\
$ $&$\scriptstyle\times(72(3Y-2)^5(3Y-4)^5(9Y-16)(9Y-2))^{-1}$\\
\hline
\end{tabular}\\
\end{center}
$ $\\

The following table lists $S_{N}(n)=\Big{(}\sum_{j=0}^{N}(-1)^{j}\Psi_{j}(n,2^n)\Big{)}^{-1}$ for certain values of $n$ and $N$. The correct digits of $(-1)^{n+1}\lambda_{n}$ are underlined.
\begin{center}
\noindent\begin{tabular}{||r || l |l|l|l||} \hline
\multicolumn{5}{||c||}{\textbf{2. Numerical values of $S_{N}(n)$}}\\
\hline
$N\backslash n$   & $2$ & $3$ & $4$ & $5$\\
\hline
& & & &\\
$5$&$\underline{0.303}59526888627_{+}$& $\underline{0.1008}4990161441_{+}$& $\underline{0.035}53807998183_{+}$& $\underline{0.012}91316972180_{+}$\\
& & & &\\
$10$&$\underline{0.30366}223674122_{+}$& $\underline{0.100884}16824324_{+}$& $\underline{0.03549}756178532_{+}$& $\underline{0.01284}672844801_{+}$\\
& & & &\\
$20$&$\underline{0.303663002}67057_{+}$& $\underline{0.100884509}98887_{+}$& $\underline{0.03549615}817871_{+}$& $\underline{0.0128437}8208157_{+}$\\
& & & &\\
$40$&$\underline{0.303663002898}81_{+}$& $\underline{0.1008845092}6615_{+}$& $\underline{0.035496159}13748_{+}$& $\underline{0.012843790}96146_{+}$\\
\hline
\end{tabular}\\
\end{center}

\begin{center}
\begin{tabular}{||r | c||r|c||}
\hline
\multicolumn{4}{||c||}{\textbf{3. The sequence $\Psi_{j}(2,4)$}, $0\leq j\leq 9$.}\\
\hline
& & & \\
$0$ & $\displaystyle 6$    &$5$&$\displaystyle\frac{2509823493}{2^4\cdot5^9\cdot13\cdot17^3}$\\
& & & \\
$1$ & $\displaystyle\frac{13}{5}$              &$6$&$\displaystyle-\frac{19855538966267}{2^6\cdot5^{11}\cdot13^2\cdot17^4}$\\
& & & \\
$2$ & $\displaystyle-\frac{11}{5^2}$   &$7$&$\displaystyle\frac{129127251417135911831}{2^8\cdot3\cdot5^{13}\cdot13^3\cdot17^5\cdot257}$\\
& & & \\
$3$ & $\displaystyle\frac{678}{5^5\cdot17}$  &$8$&$\displaystyle-\frac{662052024234553451842334613}{2^{10}\cdot5^{15}\cdot7\cdot13^4\cdot17^6\cdot257^2}$\\
& & & \\
$4$&$\displaystyle-\frac{74439}{2^2\cdot5^6\cdot17^2}$  &$9$&$\displaystyle\frac{781786025527978909165252890560522881}{2^{12}\cdot3^2\cdot5^{16}\cdot7^2\cdot13^5\cdot17^7\cdot41\cdot257^3}$\\
& & & \\
\hline
\end{tabular}\\
\end{center}

\begin{center}
\noindent\begin{tabular}{||r || c ||} \hline
\multicolumn{2}{||c||}{\textbf{4. Rational functions $\Lambda_{j}(\mathbf{X},\mathbf{Y})$}, $0\leq j\leq 3$.}\\
\hline
&\\
$0$&$2Y-2$\\
& \\
$1$&$\displaystyle\frac{XY^2+Y^2-3Y+2}{3Y-2}$\\
&  \\
$2$&$\scriptstyle(9X^2Y^4-3XY^4-12Y^4-22X^2Y^3+30XY^3+52Y^3-4X^2Y^2-60XY^2-56Y^2+8X^2Y+24XY+16Y)$\\
$ $&$\scriptstyle\times(12(3Y-2)^3)^{-1}$\\
& \\
$3$&$\scriptstyle(+243X^3Y^7-972X^2Y^7+81XY^7+1296Y^7-792X^3Y^6+2916X^2Y^6-3276XY^6-6984Y^6+3840X^3Y^5-3816X^2Y^5$\\
$ $&$\scriptstyle+1632XY^5+9288Y^5-2720X^3Y^4+10752X^2Y^4+13424XY^4-48Y^4-816X^3Y^3-13696X^2Y^3-20976XY^3$\\
$ $&$\scriptstyle-8096Y^3+1152X^3Y^2+7232X^2Y^2+11904XY^2+5824Y^2-256X^3Y-1408X^2Y-2432XY-1280Y)$\\
$ $&$\scriptstyle\times(72(3Y-2)^5(9Y-2))^{-1}$\\
\hline
\end{tabular}\\
\end{center}

The following table lists $L_{N}(s)=\Big{(}\sum_{j=0}^{N}(-1)^{j}\Lambda_{j}(s-1,2^{s-1})\Big{)}^{-1}$ for certain values of $s$ and $N$. The correct digits of $\lambda_{1}(s)$ are underlined. Apart from $\lambda_{1}(4)$, underlined digits are only those coinciding with the respective digits of the next entry. The constant $\lambda_{1}(4)$ and other eigenvalues of $\mathcal{L}_{4}$ appear in \cite{flajolet3} in connection with lattice reduction algorithm in dimension 2. The number $\lambda_{1}(4)$ is known as \emph{the Vall\'{e}e constant}  \cite{finch,flajolet4}.
\begin{center}
\noindent\begin{tabular}{||r || l |l|l|l||} \hline
\multicolumn{5}{||c||}{\textbf{5. Numerical values of $L_{N}(s)$}}\\
\hline
$N\backslash s$   & $5/2$ & $3$ & $4$ & $18$\\
\hline
& & & &\\
$5$&$\underline{0.599}17481197326_{+}$& $\underline{0.396}52695432136_{+}$& $\underline{0.199}50800589324_{+}$& $-0.0000863940590_{+}$\\
& & & &\\
$10$&$\underline{0.59908}463832611_{+}$& $\underline{0.39643}522729436_{+}$& $\underline{0.19945}914049411_{+}$& $\underline{0.00017329}459246_{+}$\\
& & & &\\
$20$&$0.59908399859453_{+}$& $0.39643461357311_{+}$& $\underline{0.19945881834}668_{+}$& $0.00017329515765_{+}$\\

\hline
\end{tabular}\\
\end{center}
$ $\\
As is predicted by the Conjecture \ref{conj3}, $\frac{1}{\lambda_{1}(4)}=\sum\limits_{j=0}^{\infty}(-1)^{j}\Lambda_{j}(3,8)$.
\begin{center}
\begin{tabular}{||r | c||r|c||}
\hline
\multicolumn{4}{||c||}{\textbf{6. The sequence $\Lambda_{j}(3,8)$}, $0\leq j\leq 9$.}\\
\hline
& & & \\
$0$ & $\displaystyle 14$    &$5$&$\displaystyle-\frac{43041548352866}{3^4\cdot5^4\cdot11^9\cdot131}$\\
& & & \\
$1$ & $\displaystyle\frac{117}{11}$              &$6$&$\displaystyle\frac{5090277810440554529}{3^4\cdot5^6\cdot11^{11}\cdot131^2}$\\
& & & \\
$2$ & $\displaystyle\frac{6280}{3\cdot11^3}$   &$7$&$\displaystyle-\frac{56022698078692317056550307}{3^6\cdot5^8\cdot11^{12}\cdot103\cdot131^3}$\\
& & & \\
$3$ & $\displaystyle-\frac{89128}{3^2\cdot11^5}$  &$8$&$\displaystyle\frac{78083190108525386193197430572483023}{2\cdot3^7\cdot5^{10}\cdot11^{14}\cdot17\cdot103^2\cdot131^4}$\\
& & & \\
$4$&$\displaystyle\frac{17097857}{3\cdot5^2\cdot11^7}$  &$9$&$\displaystyle-\frac{5692648747977502069023785944103249502105799841}{2\cdot3^7\cdot5^{12}\cdot11^{16}\cdot17^2\cdot103^3\cdot131^5\cdot293}$\\
& & & \\
\hline
\end{tabular}\\
\end{center}
$ $\\

\par\bigskip

\indent Institute of Mathematics, Department of Integrative Biology,
Universit\"{a}t f\"{u}r Bodenkultur Wien, Gregor Mendel-Stra{\ss}e 33, A-1180 Wien, Austria, \&\\
\indent Vilnius University, The Department of Mathematics and
Informatics, Naugarduko 24, Vilnius, Lithuania.\\
{\tt giedrius.alkauskas@gmail.com}\\


\begin{thebibliography}{9}
\bibitem{alkauskas}{\sc G. Alkauskas}, The Minkowski question mark function: explicit series for the dyadic period function and moments,
{\it Mathematics of Computation} {\bf 79} (269) (2010), 383--418.

\bibitem{babenko} {\sc K. I. Babenko}, A problem by Gauss, {\it Dokl. Akad.
    Nauk SSSR} {\bf 238} (5) (1978), 1021--1204; English translation: {\it
    Soviet Math. Dokl. } {\bf 19} (1) (1978), no. 1, 136--140.

\bibitem{briggs}{\sc K. Briggs} (2003), A precise computation of the
Gauss-Kuzmin-Wirsing constant, available electronically at: \url{http://keithbriggs.info/documents/wirsing.pdf}

\bibitem{calkin} {\sc N. Calkin, H. Wilf}, Recounting the rationals, {\it
    American Mathematical Monthly} {\bf 107} (2000), 360--363.

\bibitem{flajolet2} {\sc H. Daud\'{e}, Ph. Flajolet, B. Vall\'{e}e}, An average-case analysis of the Gaussian algorithm for lattice reduction,
      {\it Combinatorics, Probability and Computing} {\bf 6} (4) (1997), 397--433.

\bibitem{finch}{\sc S. R. Finch}, {\it Mathematical constants}, Encyclopedia of Mathematics and its Applications, 94. Cambridge University Press, Cambridge, 2003.

\bibitem{flajolet1} {\sc Ph. Flajolet, B. Vall\'{e}e} (1995), On the Gauss-Kuzmin-Wirsing constant, available electronically at:
\url{http://algo.inria.fr/flajolet/Publications/gauss-kuzmin.ps}


\bibitem{flajolet3} {\sc Ph. Flajolet, B. Vall\'{e}e}, Continued fraction algorithms, functional operators, and structure constants,
      {\it Theoretical Computer Science} {\bf 194} (1-2) (1998), 1--34.

\bibitem{flajolet4} {\sc Ph. Flajolet, B. Vall\'{e}e}, Continued fractions, comparison algorithms, and fine structure constants,
{\it Constructive, experimental, and nonlinear analysis (Limoges, 1999)},  53--82, CMS Conf. Proc., 27, Amer. Math. Soc., Providence, RI,  2000.

\bibitem{hensley} {\sc D. Hensley}, The number of steps in the Euclidean algorithm,
      {\it Journal of Number Theory} {\bf 49} (2) (1994), 142--182.

\bibitem{khinchin}{\sc A. Ya. Khinchin}, {\it Continued fractions}, The
    University of Chicago Press, 1964.

\bibitem{knuth}{\sc D. E. Knuth}, {\it The art of computer programming}, 2nd ed., vol 2: Seminumerical algorithms,
    Addison-Wesley, 1981.

\bibitem{zagier2}{\sc J. Lewis, D. Zagier}, Period functions and the Selberg zeta function for the modular group, {\it The mathematical beauty of physics (Saclay, 1996)},  83--97, Adv. Ser. Math. Phys., 24, {\it World Sci. Publ., River Edge, NJ}, 1997.

\bibitem{macleod}{\sc A. J. MacLeod}, High-accuracy numerical values in the Gauss-Kuz'min continued fraction problem,
{\it Computers \and Mathematics with Applications} {\bf 26} (3) (1993), 37--44.

\bibitem{mayer1}{\sc D. Mayer, G. Roepstorff}, On the relaxation time of Gauss's continued-fraction map. I. The Hilbert space approach (Koopmanism),
{\it Journal of Statistical Physics} {\bf 47} (1-2) (1987), 149--171.

\bibitem{mayer2}{\sc D. Mayer, G. Roepstorff}, On the relaxation time of Gauss' continued-fraction map. II. The Banach space approach (transfer operator method)., {\it Journal of Statistical Physics} {\bf 50} (1-2) (1988), 331--344.

\bibitem{mayer3}{\sc D. Mayer}, The thermodynamic formalism approach to Selberg's zeta function for ${\rm PSL}(2,\mathbb{Z})$, {\it Bulletin of the American Mathematical Society (New Series)} {\bf 25} (1) (1991), 55--60.

\bibitem{mayer4}{\sc D. Mayer}, On the thermodynamic formalism for the Gauss map, {\it Communications in Mathematical Physics} {\bf 130} (2) (1990), 311--333.

\bibitem{wirsing} {\sc E. Wirsing}, On the theorem of Gauss-Kusmin-L\'{e}vy and a Frobenius-type theorem for function spaces, {\it Acta Arithmetica}
{\bf 24} (1973/74), 507--528.

\bibitem{zagier1} {\sc D. Zagier} (2001), New points of view on the Selberg zeta function.

\end{thebibliography}
\end{document}